\documentclass[a4paper]{amsart}
\usepackage{amsmath, amssymb}
\usepackage{hyperref}
\usepackage{graphicx}
\usepackage[margin=3cm]{geometry}
\numberwithin{equation}{section}
\newtheorem{theorem}{Theorem}[section]

\newtheorem{lemma}{Lemma}[section]

\theoremstyle{remark}

\title[Some new inequalities for univalent functions]{Some new inequalities for univalent functions}



\begin{document}
\begin{abstract}
By use of a modified Nunokawa's lemma, we obtain some new conditions for univalence. Also, some sharp inequalities concerning univalent functions are presented.
\end{abstract}

\author[M.M. Motamedinezhad and R. Kargar] {M. M. Motamedinezhad and R. Kargar}
\address{Department of Mathematics, University of Applied Science and Technology, Tehran, Iran}
\email{motamedi@uast.ac.ir}
\address{Young Researchers and Elite Club,
Ardabil Branch, Islamic Azad University, Ardabil, Iran}
\email{rkargar@pnu.ac.ir, rkargar1983@gmail.com}

\maketitle

\section{Introduction}
Let $\Delta$ be the open unit disc on the complex plane $\mathbb{C}$. We denote by $\mathfrak{B}$ the class of functions $w(z)$ analytic in $\Delta$ with $w(0)=0$ and $|w(z)|<1$, and denote by $\mathcal{A}$ the class of all functions that are analytic and normalized in $\Delta$. A function belongs to the class $\mathcal{A}$ has the following form
\begin{equation}\label{f}
  f(z)=z+a_2z^2+\cdots+ a_n z^n+\cdots\quad (z\in\Delta).
\end{equation}
The subclass of $\mathcal{A}$ consisting of univalent functions in $\Delta$ is denoted by $\mathcal{S}$.
Also we denote by $\mathcal{S}^*$ the class of starlike functions and denote by $\mathcal{S}^*(\mu)$ the class of all starlike functions of order $\mu$, $0\leq \mu<1$. Put $\mathcal{S}^*(0)\equiv\mathcal{S}^*$.
Analytically $f\in \mathcal{S}^*(\mu)$ if, and only if ${\rm Re} \left\{zf'(z)/f(z)\right\}>\mu$.
Also, we say that the function $f$ is convex of order $\mu$ in $\Delta$, if and only if $zf'(z)\in \mathcal{S}^*(\mu)$. The class of convex functions of order $\mu$ is denoted by $\mathcal{K}(\mu)$ and the class of convex functions is denoted by $\mathcal{K}\equiv\mathcal{K}(0)$.
Let $\mathcal{M}(\delta)$ be the subclass
of $\mathcal{A}$ consisting of functions $f$ which satisfy
\begin{equation*}
  {\rm Re} \left\{\frac{zf'(z)}{f(z)}\right\}< \delta\quad(z\in \Delta),
\end{equation*}
where $\delta>1$. Also, we say that $f\in \mathcal{N}(\delta)$ if and only if $zf'(z)\in \mathcal{M}(\delta)$. The class $\mathcal{M}(\delta)$ and $\mathcal{N}(\delta)$ for $1<\delta\leq4/3$ were introduced by Uralegaddi $et$ $al.$ \cite{Ura}. Also, the class $\mathcal{N}(\delta)$ including locally univalent functions was studied extensively by Kargar $et$ $al$., \cite{kargarJMAA} (see also, \cite{Mah}) and they proved that $\mathcal{N}(3/2)\subset \mathcal{S}$. Thus, $\mathcal{M}(3/2)$ and $\mathcal{N}(3/2)$ are a subclass of univalent functions.
A function $f\in \mathcal{A}$ is said to be close--to--convex, if there exists
a convex function $g$ such that
\begin{equation*}
  {\rm Re}\left\{\frac{f'(z)}{g'(z)}\right\}>0\quad(z\in\Delta).
\end{equation*}
This class was introduced by Kaplan in 1952 \cite{Kaplan} and we denote by $\mathcal{CK}$.
If we take $g(z)\equiv z$, then the class $\mathcal{CK}$ becomes the Noshiro--Warschawski class as follows
\begin{equation*}
  \mathcal{C}:=\{f\in \mathcal{A}: {\rm Re}\{f'(z)\}>0, \, z\in\Delta\}.
\end{equation*}
Note that $\mathcal{C}\subset\mathcal{S}$, by the basic Noshiro--Warschawski lemma \cite[\S 2.6]{Duren}.

In 1977, Chichra \cite{Chichra} studied the class of all functions whose derivative has positive real part in the unit disc $\Delta$. Indeed, he denoted by $\mathcal{F}_\gamma$ the class of functions $f\in\mathcal{A}$ which satisfying the following inequality
\begin{equation*}
  {\rm Re}\left\{f'(z)+\gamma zf''(z)\right\}>0\quad(z\in\Delta),
\end{equation*}
where $\gamma\geq0$, and showed that $\mathcal{F}_\gamma\subset \mathcal{S}$. Also, he proved that if $f\in\mathcal{F}_\gamma $ and ${\rm Re}\{\gamma\}\geq0$, then ${\rm Re}\{f'(z)\}>0$ in $\Delta$. Recent result, also was obtained by Lewandowski $et$ $al.$, see \cite{lewa et al}.
On the other hand, Gao and Zhou \cite{gaozhou} considered the class $R(\beta,\gamma)$ as follows:
\begin{equation*}
  R(\beta,\gamma)=\left\{f\in\mathcal{A}:{\rm Re}\left\{f'(z)+\gamma zf''(z)\right\}>\beta,\ \ \gamma>0,\, \beta<1,\, z\in \Delta\right\}.
\end{equation*}
They found the extreme points of $ R(\beta,\gamma)$, some sharp bounds of certain
linear problems, the sharp bounds for ${\rm Re}\{f'(z)\}$ and ${\rm Re}\{f(z)/z\}$ and determined the number
$\beta(\gamma)$ such that $R(\beta,\gamma)\subset\mathcal{S}^*$, where $\gamma$ is certain fixed number in $[1,\infty)$. Note that a generalization of the class $R(\beta,\gamma)$ was studied by Srivastava et al. in \cite{sridorzap}.

To prove of our main results we need the following lemma.
\begin{lemma}\label{lem Nuno}
  {\rm(}Simple generalization of Nunokawa's lemma \cite{Nuno}{\rm)} Let $p(z)$ be an analytic function in $|z| <1$ of the form
  \begin{equation*}
    p(z)=1+\sum_{n=m}^{\infty} c_n z^n\quad(c_m\neq0),
  \end{equation*}
  with $p(z) \neq 0$ in $|z| <1$. If there exists a point $z_0$, $|z_0| <1$, such that
  \begin{equation*}
    {\rm Re} \{p(z)\}>0\quad for\quad |z|<|z_0|
  \end{equation*}
  and
    \begin{equation*}
    {\rm Re} \{p(z)\}=0, \quad a=|p(z_0)|\neq0,
  \end{equation*}
  then we have
  \begin{equation*}
    \frac{z_0p'(z_0)}{p(z)}=i k,
  \end{equation*}
  where $k$ is real number and
  \begin{equation*}\label{l>}
    k\geq \frac{m}{2}\left(a+\frac{1}{a}\right)\geq m\geq1\ \ when \ \ p(z_0)=i a
  \end{equation*}
  and
    \begin{equation*}\label{l<}
    k\leq -\frac{m}{2}\left(a+\frac{1}{a}\right)\leq -m\leq-1\ \ when \ \ p(z_0)=-ia.
  \end{equation*}
  \end{lemma}
In the present paper, motivated by the works \cite{NS1} and \cite{NS2} we shall determine some new sufficient conditions for close--to--convexity. Also, some another applications of Nunokawa's lemma concerning the class $\mathcal{M}(3/2)$, $\mathcal{N}(3/2)$ and $\mathcal{S}^*(1/2)$ are presented.


\section{Application of Nunokawa's lemma}\label{Appl. Nunokawa's lemma}

In this section, we present some applications of Nunokawa's lemma. The first result is as follows which it introduces new conditions for close--to--convexity, hence univalence.
\begin{theorem}
  Let $f$ be of the form \eqref{f}. If $f$ satisfies
  \begin{equation}\label{eq. f'+zf''<01}
    {\rm Re}\left\{f'(z)+zf''(z)\right\}\leq 1\quad(z\in\Delta),
  \end{equation}
  or
  \begin{equation}\label{eq. f'+zf''<1}
    {\rm Re}\left\{f'(z)+zf''(z)\right\}> 1\quad(z\in\Delta),
  \end{equation}
  then
    \begin{equation*}
   {\rm Re}\left\{\frac{1}{f'(z)}\right\}>\frac{1}{2}\quad(z\in\Delta)
 \end{equation*}
 and
  \begin{equation}\label{|f'(z)-1|<1}
  |f'(z)-1|<1 \quad (z\in\Delta).
  \end{equation}
  The inequality \eqref{|f'(z)-1|<1} means that $f$ is close--to--convex (hence univalent).
\end{theorem}
\begin{proof}
  Let us define the function $p(z)$ as follows
  \begin{equation}\label{p(z) f'(z)}
    p(z)=\frac{2}{f'(z)}-1\quad(z\in\Delta),
  \end{equation}
  where $f\in \mathcal{A}$. Then $p$ is analytic function in $\Delta$ and $p(0)=1$. With a simple computation, \eqref{p(z) f'(z)} implies that
  \begin{equation}\label{f' p}
    f'(z)=\frac{2}{1+p(z)}\quad(z\in\Delta)
  \end{equation}
  and
  \begin{equation}\label{1+zf'' p(Z)}
    1+\frac{zf''(z)}{f'(z)}=1-\frac{zp'(z)}{1+p(z)}\quad(z\in\Delta).
  \end{equation}
  Now, from \eqref{f' p} and \eqref{1+zf'' p(Z)}, we get
  \begin{align}\label{f'+zf''}
    f'(z)+zf''(z)&=f'(z)\left(1+\frac{zf''(z)}{f'(z)}\right)\\
    &=\left(\frac{2}{1+p(z)}\right)\left(1-\frac{zp'(z)}{p(z)}
    \frac{p(z)}{1+p(z)}\right).\nonumber
  \end{align}
  Suppose that there exists a point $z_0\in\Delta$ so that
  \begin{equation*}
    |f'(z)-1|<1\quad(|z|<|z_0|)
  \end{equation*}
  and
    \begin{equation*}
    |f'(z_0)-1|=1.
  \end{equation*}
  Then applying \eqref{|f'(z)-1|<1}, we obtain ${\rm Re}\{p(z)\}>0$ for $|z|<|z_0|$ and ${\rm Re}\{p(z_0)\}=0$. Also $p(z_0)\neq0$. Now, by using of Nunokawa's lemma, we have
  \begin{equation*}\label{z0p'z0}
    \frac{z_0p'(z_0)}{p(z_0)}=ik,
  \end{equation*}
  where
  \begin{equation}\label{kgeq1}
    k\geq \frac{1+a^2}{2a}\quad{\rm when}\quad p(z_0)=ia
  \end{equation}
  and
  \begin{equation*}\label{kleq1}
    k\leq -\frac{1+a^2}{2a}\quad{\rm when}\quad p(z_0)=-ia.
  \end{equation*}
  First, we investigate the case $p(z_0)=ia$. From \eqref{f'+zf''}, we have
  \begin{align*}
    {\rm Re}\left\{f'(z_0)+z_0f''(z_0)\right\}
    &={\rm Re}\left\{\left(\frac{2}{1+p(z_0)}\right)\left(1-\frac{z_0p'(z_0)}{p(z_0)}
    \frac{p(z_0)}{1+p(z_0)}\right)\right\}\\
    &={\rm Re}\left\{\left(\frac{2}{1+ia}\right)
    \left(1-ik.\frac{ia}{1+ia}\right)\right\}\\
    &={\rm Re}\left\{\frac{2(1+ka)+2ia}{(1+ia)^2}\right\}\\
    &=\frac{2(1+ka)(1-a^2)+4a^2}{(1+a^2)^2}=:k(a).
  \end{align*}
  Since $a>0$, the following two cases arise.\newline
  {\bf Case 1.} Let $0<a<1$. Then by \eqref{kgeq1}, we get
  \begin{equation*}\label{h}
    k(a)\geq \frac{3+2a^2-a^4}{(1+a^2)^2}=:h(a).
  \end{equation*}
  It is easy to see that $h$ is decreasing and thus
 \begin{equation*}
   h(a)>h(1)=1.
 \end{equation*}
 But, this is contradictory with assumption \eqref{eq. f'+zf''<01}. Therefore we have
 \begin{equation*}
   {\rm Re}\{p(z)\}>0\quad(z\in\Delta)
 \end{equation*}
 or equivalently
  \begin{equation*}
   {\rm Re}\left\{\frac{1}{f'(z)}\right\}>\frac{1}{2}\quad(z\in\Delta).
 \end{equation*}
 Finally, from \eqref{p(z) f'(z)} we have
\begin{equation*}
  |f'(z)-1|=\left|\frac{1-p(z)}{1+p(z)}\right|<1\quad(z\in\Delta).
\end{equation*}
{\bf Case 2.} Let $a\geq1$. Then by \eqref{kgeq1} we have $k(a)\leq h(a)$. Because $h$ is decreasing function and $a\geq1$ we get $k(a)\leq h(a)\leq h(1)=1$. This contradicts the hypothesis \eqref{eq. f'+zf''<1} and thus we have the inequality \eqref{|f'(z)-1|<1}. Note that the proof of the case $p(z_0)=-ia$ is similar and therefore we omit the details. The proof of this theorem here ends.
\end{proof}
\begin{theorem}
  Let $f$ be of the form \eqref{f} and satisfies
  \begin{equation}\label{re zf' f 32}
    {\rm Re}\left\{\frac{zf'(z)}{f(z)}\right\}< \frac{3}{2}\quad(z\in \Delta).
  \end{equation}
  Then
    \begin{equation*}
   {\rm Re}\left\{\frac{z}{f(z)}\right\}>\frac{1}{2}\quad(z\in\Delta)
 \end{equation*}
 and
  \begin{equation}\label{f z - 1 1}
    \left|\frac{f(z)}{z}-1\right|<1\quad(z\in \Delta).
  \end{equation}
  The results are sharp for the function $z\mapsto z+z^2$.
\end{theorem}
\begin{proof}
  We define the function $p(z)$ as follows
  \begin{equation}\label{p(z) fz z}
    p(z):=\frac{2z}{f(z)}-1,
  \end{equation}
  with $p(0)=1$. A simple calculation gives us
  \begin{equation}\label{zf' 1-zp}
    \frac{zf'(z)}{f(z)}=1-\frac{zp'(z)}{p(z)}\frac{p(z)}{1+p(z)}.
  \end{equation}
   If there is a point $z_0\in\Delta$ such that
  \begin{equation*}
    \left|\frac{f(z)}{z}-1\right|<1\quad(|z|<|z_0|)
  \end{equation*}
  and
    \begin{equation*}
    \left|\frac{f(z_0)}{z_0}-1\right|=1,
  \end{equation*}
  then by \eqref{f z - 1 1}, we get ${\rm Re}\{p(z)\}>0$ when $|z|<|z_0|$, ${\rm Re}\{p(z_0)\}=0$ and $p(z_0)\neq0$. Now, by the Nunokawa lemma, we have
  \begin{equation*}\label{z0p'z0 2}
    \frac{z_0p'(z_0)}{p(z_0)}=ik,
  \end{equation*}
  where
  \begin{equation}\label{kgeq1 2}
    k\geq \frac{1+a^2}{2a}\quad{\rm when}\quad p(z_0)=ia
  \end{equation}
  and
  \begin{equation}\label{kleq1 2}
    k\leq -\frac{1+a^2}{2a}\quad{\rm when}\quad p(z_0)=-ia.
  \end{equation}
  We investigate the case \eqref{kgeq1 2} and let $a>0$. The proof of another case \eqref{kleq1 2} is similar. From \eqref{zf' 1-zp} and \eqref{kgeq1 2}, we obtain
  \begin{align*}
    {\rm Re}\left\{\frac{z_0f'(z_0)}{f(z_0)}\right\}
    &={\rm Re}\left\{1-\frac{z_0p'(z_0)}{p(z_0)}\frac{p(z_0)}{1+p(z_0)}\right\}\\
    &={\rm Re}\left\{1+\frac{ka}{1+ia}\right\} \\
    &= 1+\frac{ka}{1+a^2}\\
    &\geq \frac{3}{2}
  \end{align*}
  and this contradicts the hypothesis \eqref{re zf' f 32}. Ergo, we have
  \begin{equation*}
   {\rm Re}\{p(z)\}>0\quad(z\in\Delta)
 \end{equation*}
 or equivalently
  \begin{equation*}
   {\rm Re}\left\{\frac{z}{f(z)}\right\}>\frac{1}{2}\quad(z\in\Delta).
 \end{equation*}
  Finally, \eqref{p(z) fz z} implies that
  \begin{equation*}
    \left|\frac{f(z)}{z}-1\right|=\left|\frac{1-p(z)}{1+p(z)}\right|<1\quad(|z|<1)
  \end{equation*}
  and concluding the proof.
\end{proof}
It is well--known that (see \cite{Marx, Str}) if
  \begin{equation*}
    {\rm Re}\left\{1+\frac{zf''(z)}{f'(z)}\right\}> 0\quad(z\in \Delta),
  \end{equation*}
then ${\rm Re}\{\sqrt{f'(z)}\}>1/2$ and ${\rm Re}\{f(z)/z\}>1/2$. Alos, these estimates are sharp. Now by using the Alexander theorem, we have the following.
\begin{theorem}
    Let $f$ be of the form \eqref{f} and satisfies
  \begin{equation*}
    {\rm Re}\left\{1+\frac{zf''(z)}{f'(z)}\right\}< \frac{3}{2}\quad(z\in \Delta).
  \end{equation*}
  Then
    \begin{equation*}
   {\rm Re}\left\{\frac{1}{f'(z)}\right\}>\frac{1}{2}\quad(z\in\Delta)
 \end{equation*}
 and
  \begin{equation*}
    \left|f'(z)-1\right|<1\quad(z\in \Delta).
  \end{equation*}
  The results are sharp for the function $z\mapsto z+z^2/2$.
\end{theorem}
Finally, using the Nunokawa lemma we obtain the following well--known result, see \cite[p. 73]{Duren}.
\begin{theorem}
  Let $f\in \mathcal{A}$ be starlike function of order $1/2$.
  Then
    \begin{equation*}
   {\rm Re}\left\{\frac{f(z)}{z}\right\}>\frac{1}{2}\quad(z\in\Delta)
 \end{equation*}
 and
  \begin{equation*}
    \left|\frac{z}{f(z)}-1\right|<1\quad(z\in \Delta).
  \end{equation*}
  The results are sharp for the function $z\mapsto \frac{z}{1+z}$.
\end{theorem}
\begin{proof}
  It is enough to define $p(z)+1=2f(z)/z$ and to apply the Nunokawa lemma, then we get the desired result.
\end{proof}


\end{document}